\documentclass[11pt,a4paper]{amsart}
\usepackage[all]{xy}
\SelectTips{cm}{}
\calclayout

\usepackage{amscd,mathtools,amssymb,amsopn,amsmath,amsthm,graphics,mathrsfs,accents}
\usepackage[colorlinks=true,linkcolor=red,citecolor=blue]{hyperref}
\usepackage[mathscr]{euscript}
\usepackage[usenames]{color}

\newcommand{\rt}{\rightarrow}

\newcommand{\lrt}{\longrightarrow}

\newcommand{\st}{\stackrel}

\newcommand{\Z}{\mathbb{Z} }

\newcommand{\Gp}{\mathsf{Gp} }

\def\Ext{\operatorname{{\mathsf{Ext}}}}

\def\hom{\operatorname{{\mathsf{Hom}}}}

\def\Tor{\operatorname{\mathsf{Tor}}}
\def\hom{\operatorname{\mathsf{Hom}}}
\def\H{\operatorname{\mathsf{H}}}

\def\syz{\mathsf{\Omega}}

\def\Gpd{\operatorname{\mathsf{Gpd}}}

\def\mod{\operatorname{\mathsf{mod}}}

\newtheorem{theorem}{Theorem}[section]

\newtheorem{cor}[theorem]{Corollary}

\newtheorem{lem}[theorem]{Lemma}
\newtheorem{prop}[theorem]{Proposition}

\theoremstyle{definition}
\newtheorem{dfn}[theorem]{Definition}

\newtheorem{rem}[theorem]{Remark}
\newtheorem{s}[theorem]{}

\theoremstyle{plain}

\theoremstyle{definition}

\numberwithin{equation}{section}

\begin{document}

\title[On the vanishing of $\Ext$ and $\Tor$]
{On the vanishing of $\Ext$ and $\Tor$}
\author[Bahlekeh and Salarian]{Abdolnaser Bahlekeh and  Shokrollah Salarian}

\address{Department of Mathematics, Gonbad-Kavous University, Postal Code:4971799151, Gonbad-Kavous, Iran}
\email{bahlekeh@gonbad.ac.ir}

\address{Department of Pure Mathematics, Faculty of Mathematics and Statistics,
University of Isfahan, P.O.Box: 81746-73441, Isfahan,
 Iran and \\ School of Mathematics, Institute for Research in Fundamental Science (IPM), P.O.Box: 19395-5746, Tehran, Iran}
 \email{Salarian@ipm.ir}


\thanks{}

\begin{abstract}
This paper contains two theorems concerning the vanishing of natural transformations of (co)homology functors. Precisely, assume that $R$ is a right noetherian ring and $f: M\rt N$ is a morphism of finitely generated right $R$-modules. The first theorem proves that the natural transformation $\Ext^1(f, -)$ vanishes over the category of finitely generated right $R$-modules if and only if $\Tor_1(f, -)$ vanishes over the category of finitely generated left $R$-modules. As a corollary of this result, we establish that $\Ext^1(f, -)$ is epic if and only if $\Tor_1(f, -)$ is monic. The second theorem shows that if $R$ is left and right noetherian and $M, N$ are Gorenstein projective, then the natural transformations $\Tor_1(f, -)$, $\Ext^1(-, f)$ and $\Ext^1(f, -)$ vanish over the category of finitely generated Gorenstein projective modules, simultaneously. This, in particular, yields that over Gorenstein projective modules, the notions of phantom morphisms and $\Ext$-phantom morphisms coincide. Also, it is proved that if $R$ is $n$-Gorenstein, then for any integer $i>n$, the natural transformations $\Ext^{i}(f, -)$, $\Ext^{i}(-, f)$ and $\Tor_{i}(f, -)$ vanish over finitely generated modules,  simultaneously. As an interesting consequence, we show that under the same assumptions, $\Ext^i(-, f)$ is epic (resp. monic) if and only if $\Ext^i(f, -)$ is monic (resp. epic) if and only if $\Tor_i(f, -)$ is epic (resp. monic).
 \end{abstract}
\maketitle


\section{Introduction}
Throughout the paper $R$ is a right noetherian ring, and unless specified otherwise, all modules are assumed to be finitely generated right $R$-modules. We use $\mod R$ (resp. $\mod R^{op}$) to denote the category of all finitely generated right (resp. left) $R$-modules.

Assume that $f: M\rt N$ is a morphism in $\mod R$. This paper deals with the vanishing of the natural transformations $\Tor_1(f, -)$, $\Ext^1(f, -)$ and $\Ext^1(-, f)$. The study of the vanishing of such natural transformations stems from the notion of phantom morphisms.  A given morphism $g: X\rt Y$ of (not necessarily finitely generated) right $R$-modules is called a phantom morphism, if the natural transformation $\Tor_1(g, -):\Tor^R_1(X, -)\rt\Tor^R_1(Y, -)$  is zero, or equivalently, the pullback of any short exact sequence along $g$ is pure exact.  Similarly, $g$ is called an $\Ext$-phantom morphism, if for any finitely presented right $R$-module $L$, the induced morphism of abelian groups $\Ext^1(L, g):\Ext^1_R(L, X)\rt\Ext^1_R(L, Y)$ is zero.
The study of phantom morphisms has its roots in topology in the study of maps between CW-complexes \cite{mc}. A map $h: X\rt Y$ between CW-complexes is said to be a phantom map if its restriction to each skeleton $X_n$ is null homotopic. For the historical information on phantom maps see Remark \ref{s100}.

This paper proves two theorems that compare the vanishing of natural transformations of (co)homology functors, over finitely generated modules. Precisely, our first theorem reads as follows.
\begin{theorem}\label{thm} For a given morphism  $f:M\rt N$ in $\mod R$, the following assertions are equivalent:
\begin{enumerate}\item $\Ext^1(f, -)$ vanishes over $\mod R$. \item  $\Tor_1(f, -)$ vanishes  over $\mod R^{op}$.
\end{enumerate}
\end{theorem}

It is worth noting that this theorem can be viewed as the morphism version of \cite[Proposition 3.2.12]{ej} and \cite[Corollary 3.57]{rot}, asserting that over a noetherian ring, a finitely generated flat module is projective. Using Theorem \ref{thm}, we prove that the natural transformation $\Ext^1(f, -)$ is epic if and only if  $\Tor_1(f, -)$ is monic, and also, if $\Tor_1(f, -)$ is epic, then  $\Ext^1(f, -)$ will be monic, see Corollary \ref{cc}. We should point out that the dimension shifting argument implies that all the assertions in Theorem \ref{thm} and Corollary \ref{cc} still remain true for the natural transformations $\Ext^i(f, -)$ and $\Tor_i(f, -)$ for any integer $i\geq 1$, see Corollaries \ref{i1} and \ref{ii1}.

Our second  theorem indicates  that if $f:M\rt N$ is a morphism in $\Gp(R)$, the category of all finitely generated Gorenstein projective right $R$-modules, then the natural transformations $\Tor_1(f, -)$, $\Ext^1(f, -)$ and $\Ext^1(-, f)$ vanish, simultaneously. Indeed, we prove the result below.
\begin{theorem}\label{thm2}Let $R$ be a left and right  noetherian ring and let $f: M\rt N$ be a morphism in $\Gp(R)$. Then the following are equivalent:
\begin{enumerate}\item  $\Tor_1(f, -)=0$ over $\Gp(R^{op})$. \item $\Ext^1(f, -)=0$ over $\Gp(R)$. \item $\Ext^1(-, f)=0$ over $\Gp(R)$.
\end{enumerate}
\end{theorem}
It is known that over the category of Gorenstein projective modules, projective objects coincide with the injective objects. The above theorem, not only ensures the validity of the morphism version of this fact but also shows that in this setting the notions of phantom morphisms and $\Ext$-phantom morphisms coincide. Moreover,  Theorem \ref{thm2} enables us to show that, with the same assumptions, the natural transformations $\Ext^{1}(f, -)$, $\Ext^{1}(-, f)$ and $\Tor_{1}(f, -)$ are equivalence over Gorenstein projective modules, simultaneously, see Corollary \ref{gor}. Also, as an interesting result, we show that if  $R$ is an $n$-Gorenstein ring,  then for a given morphism  $f: M\rt N$ in $\mod R$ and an integer $i>n$, the natural transformations $\Ext^{i}(f, -)$, $\Ext^{i}(-, f)$ and $\Tor_{i}(f, -)$ vanish, simultaneously, see Proposition \ref{co}. As an application of this result, it is proved that, under the same assumptions,
$\Ext^i(-, f)$ is epic (resp. monic) if and only if $\Ext^i(f, -)$ is monic (resp. epic) if and only if $\Tor_i(f, -)$ is epic (resp. monic), see Corollaries \ref{ccc} and \ref{cccc}. These results should be compared with the ones in \cite[Section 2]{mao1}.


\section{Proof of theorems}
This section is devoted to give the proofs of Theorems \ref{thm} and \ref{thm2}, which compare vanishing of the natural transformations $\Tor_1(f, -)$, $\Ext^1(f, -)$ and $\Ext^1(-, f)$, where $f: M\rt N$ is a morphism in $\mod R$. Some consequences of these theorems are explored. It is also proved that if $R$ is an $n$-Gorenstein ring, then for any integer $i>n$, the natural transformations $\Ext^{i}(f, -)$, $\Ext^{i}(-, f)$ and $\Tor_{i}(f, -)$ vanish over (finitely generated) modules, simultaneously.\\

{\bf Proof of Theorem \ref{thm}.} $(1)\Rightarrow (2)$ By our hypothesis, the natural transformation $\Ext^1(f, -)$ vanishes over the category $\mathsf{mod}R$. We would like to show that $\Tor_1(f, -)=0$ on $\mathsf{mod}R^{op}$. This will be done by showing that the morphism $f$ factors through a projective $R$-module.  Take a syzygy sequence $\eta: 0\rt\syz N\rt P\rt N\rt 0$.   By the assumption, the induced morphism of abelian groups $\Ext^1(f, \syz N):\Ext^1_{R}(N, \syz N)\lrt\Ext^1_{R}(M, \syz N)$ is zero. It is known that $\Ext^1_{R}(N, \syz N)$ (resp. $\Ext^1_{R}(M, \syz N)$) can be naturally identified with the equivalence classes of extensions of $\syz N$ by $N$ (resp. $M$), see for example \cite[Theorem 7.30]{rot}. Consequently, the morphism  $\Ext^1(f, \syz N):\Ext^1(N, \syz N)\lrt\Ext^1(M, \syz N)$ sending each object $\epsilon$ to its pull-back along $f$, is zero. This yields that the upper row in the following commutative diagram of $R$-modules, splits.
\[\xymatrix{\eta f:0\ar[r]& \syz N \ar[r] \ar@{=}[d] & T \ar[r]^{\pi'} \ar[d]_{f'} & M \ar[d]_{f}
\ar[r]&0\\ \eta: 0\ar[r]& \syz N \ar[r] & P \ar[r]^{\pi} & N\ar[r]&0.}\]So, taking an $R$-homomorphism $h:M\rt T$ with $\pi'h=1_M$, one deduces that $f=\pi f'h$. This means that $f$ factors through the projective module $P$, as desired. \\
$(2)\Rightarrow (1)$ Suppose that the natural transformation $\Tor_1(f, -)=0$ on $\mathsf{mod} R^{op}$. For a given  $R$-module $X$, we must prove that  $\Ext^1(f, X): \Ext^1_R(N, X)\lrt\Ext^1_R(M, X)$ is zero. Consider the following commutative diagram of abelian groups:
\begin{equation}\label{111}
\xymatrix{ \hom_{\Z}(\Ext^1_R(M, X), Q/{\Z})\ar[d]\ar[r]^{\alpha}& \hom_{\Z}(\Ext^1_R(N, X), Q/{\Z})\ar[d]\\  \Tor_1^R(M, \hom_{\Z}(X, Q/{\Z}))\ar[r]^{\beta}& \Tor_1^R(N, \hom_{\Z}(X, Q/{\Z})).}
\end{equation}
As $M$ and $N$ are finitely generated $R$-modules and $R$ is right noetherian, the vertical morphisms are isomorphism, see \cite[Theorem 3.2.13]{ej}. Since every module is a direct limit of finitely generated modules, and for any $R$-module $F$, the functor $\Tor_1(F, -)$ commutes with direct limits, our assumption forces the morphism $\beta$ to be zero. Consequently, the same will be true for $\alpha$, and so, $\Ext^1(f, X): \Ext^1_R(N, X)\lrt\Ext^1_R(M, X)$ is zero, as well. Thus the proof is finished. \ \ \ \ \ \ \ \  \ \ \ \ \ \ \ \ \ \ \ \ \  \ \ \ \ \ \  \ \ \ \ \ \ \  \ \ \ \ \ \ \ \ \  \ \  \ \ \  \ \ \ \ \ \ \ \ \ \ \ \ \ \ $\Box$\\

Following \cite{fght}, a given morphism of $R$-modules $f: M\rt N$ is said to be a projective morphism provided that for any $R$-module $X$, the morphism of abelian groups $\Ext^1(f, X): \Ext^1_R(N, X)\lrt\Ext^1_R(M, X)$ is zero. An injective morphism is defined dually.

A classical result in homological algebra states that a finitely presented flat module is projective, see for example \cite[Proposition 3.2.12]{ej}. Theorem \ref{thm} says that the morphism version of this result is true, as well.

Assume that $f: M\rt N$ is a morphism in $\mod R$ and $i\geq 2$ an integer. By the dimension shifting argument, we see that $\Ext^i(f, -)\cong\Ext^1(\syz^{i-1}f, -)$ and $\Tor_i(f, -)\cong\Tor_1(\syz^{i-1}f, -)$. These combined with Theorem \ref{thm} yield the result below, which guarantees the validity of the morphism version of the fact that the projective dimension of any module of finite flat dimension is also finite.
\begin{cor}\label{i1}For a given morphism $f:M\rt N$ in $\mod R$ and $i\geq 1$, the following statements are equivalent:
\begin{enumerate}\item $\Ext^i(f, -)$ vanishes over $\mod R$. \item  $\Tor_i(f, -)$ vanishes over $\mod R^{op}$.
\end{enumerate}
\end{cor}

\begin{dfn}
Let $f: M\rt N$ be a morphism in $\mod R$ and $i\geq 1$ and integer.  We say that the natural transformation  $\Tor_i(f, -)$ is monic
(resp. epic),  if for any module $X$ in $\mod R^{op}$,  the induced morphism of abelian groups $\Tor_i(f, X): \Tor^R_i(M, X)\lrt\Tor_i^R(N, X)$  is a monomorphism (resp. an epimorphism). Such a morphism is called a $\Tor_i$-monomorphism (resp. $\Tor_i$-epimorphism), see \cite[Definition 2.1]{mao1}. The case for the natural transformations $\Ext^i(f, -)$ and $\Ext^i(-, f)$ is defined similarly.
\end{dfn}

As a consequence of Theorem \ref{thm}, we state the following result.
\begin{cor}\label{cc}For a given morphism  $f: M\rt N$ in $\mod R$, the following assertions hold:
\begin{enumerate}
\item $\Ext^1(f, -)$ is epic if and only if  $\Tor_1(f, -)$ is monic. \item If  $\Tor_1(f, -)$ is epic, then  $\Ext^1(f, -)$ is monic.
\end{enumerate}
\end{cor}
\begin{proof}
(1) We just prove the `only if' part, the reverse statement can be shown analogously. Take an epimorphism $\pi: P\rt N$ with $P$ a projective $R$-module. So we will have the following commutative diagram with the exact row:
\begin{equation}\label{ee}
\xymatrix{&  & M \ar[r]^{f} \ar[d]^{[1~~0]^t} & N \ar@{=}[d]\\ 0\ar[r]&K\ar[r]^{h}  & M\oplus P \ar[r]^{f'} & N\ar[r]& 0,}
\end{equation}
where $f'=[f~~\pi]$. For a given $R$-module $X$, one may apply the functor $\Ext^1_R(-, X)$ and get the commutative diagram of abelian groups
\[\xymatrix{ \Ext^1_R(N, X)\ar[r]^{\Ext^1(f', X)}\ar@{=}[d] & \Ext^1_R(M\oplus P, X) \ar[r]^{\Ext^1(h, X)} \ar[d]^{\cong} & \Ext^1_R(K, X)\\ \Ext^1_R(N, X)\ar[r]^{\Ext^1(f, X)}  & \Ext^1_R(M, X) & & ,}\]where the upper row is exact. Since $\Ext^1(f, X)$ is an epimorphism, the same is true for $\Ext^1(f', X)$, implying that $\Ext^1(h, X)=0$. Thus, by Theorem \ref{thm}, the natural transformation {$\Tor_1(h, -)$} vanishes over $\mathsf{mod}R^{op}$. Assume that $Y$ is an $R^{op}$-module. Applying the functor $\Tor_1^R(-, Y)$ to the former diagram, yields the following commutative diagram:
\[\xymatrix{  & \Tor_1^R(M, Y) \ar[r]^{\Tor_1(f, Y)} \ar[d]^{\cong} & \Tor_1^R(N, Y) \ar@{=}[d]\\ \Tor_1^R(K, Y)\ar[r]^{\Tor_1(h, Y)}  & \Tor_1^R(M\oplus P, Y) \ar[r]^{\Tor_1(f', Y)} & \Tor_1^R(N, Y) ,}\]where the bottom row is exact. Since $\Tor_1(h, Y)=0$, the morphism $\Tor_1(f', Y)$ will be a monomorphism, and so, the same is true for $\Tor_1(f, Y)$, as needed.\\ (2) Assume that $X\in\mod R$ is given. Since $\hom_{\Z}(X, Q/{\Z})$ is the direct limit of finitely generated $R^{op}$-modules and $\Tor$ functor commutes with direct limit,  our hypothesis ensures that the morphism $\beta$ in Diagram \ref{111}, is an epimorphism. So, the same is true for the morphism $\alpha$, and then,  $\Ext^1(f, X)$ is a monomorphism, as required.
\end{proof}

Assume that $i\geq 1$ is an integer. Replacing $\Ext^1(-, X)$ (resp. $\Tor_1(-, Y))$ with $\Ext^i(-, X)$ (resp. $\Tor_i(-, Y))$ in the proof of Corollary \ref{cc} and making use of Corollary \ref{i1}, would give us the following result.
\begin{cor}\label{ii1} Let $f: M\rt N$ be a  morphism in $\mod R$  and  $i\geq 1$ an integer. Then the following statements hold:
\begin{enumerate}
\item $\Ext^i(f, -)$ is epic if and only if  $\Tor_i(f, -)$ is monic. \item If  $\Tor_i(f, -)$ is epic, then  $\Ext^i(f, -)$ is monic.
\end{enumerate}
\end{cor}

\begin{s}{\sc Gorenstein projective modules.}\\ Assume that $R$ is a left and right noetherian ring.
A given (finitely generated) $R$-module $M$ is called Gorenstein projective, if it is the zeroth  syzygy of an acyclic complex of (finitely generated) projective $R$-modules $\rm{P}^{\bullet}$ that remains exact under $\hom_R(-, Q)$ for any (finitely generated) projective $R$-module $Q$. In this case, the sequence $\rm{P}^{\bullet}$ is called a complete projective resolution of $M$.

The concept of (finitely generated) Gorenstein projective modules, which is a refinement of projective modules, originated in the works of Auslander and Bridger \cite{ab}, who introduced modules of G-dimension zero. Indeed, a finitely generated $R$-module $M$ is said to have G-dimension zero, if it is reflexive (i.e. the natural map $M\rt\hom_{R^{op}}(M^*, R)$ is an isomorphism), and $\Ext^i_R(M,  R)=0=\Ext^i_{R^{op}}(M^*, R)$ for all $ i>0$, where $(-)^*=\hom_R(-, R)$. It is well known that, over the left and right noetherian ring, a finitely generated module is Gorenstein projective if and only if it has G-dimension zero, see \cite[Lemma 2.4]{am} and also \cite[Proposition 10.2.6]{ej}.  Enochs and Jenda \cite{ej1} generalized the notion of Gorenstein projectivity to not necessarily finitely generated modules over arbitrary rings. Gorenstein projective modules have found  important applications in many areas  including commutative algebra, algebraic geometry, singularity theory, and relative homological algebra
\end{s}

In the remainder of the paper, $R$ is assumed to be a left and right noetherian ring. Also, the category of all (finitely generated) Gorenstein projective $R$-modules will be denoted by $\Gp(R)$.

Now we are ready to give the proof of Theorem \ref{thm2}. First, we show that for a given morphism of Gorenstein projective $R$-modules $f: M\rt N$, the natural transformation $\Ext^1(f, -)$ vanishes over $\Gp(R)$ if and only if the same is true for $\Ext^1(-, f)$. This, in particular, examines  the validity of the morphism version of the known fact
that over the category of Gorenstein projective modules, injectives, and projectives are the same.

\begin{prop}\label{prop}Let $f: M\rt N$ be a morphism in $\Gp(R)$. Then the natural transformations $\Ext^1(f, -)$ and $\Ext^1(-, f)$ vanish on  $\Gp(R)$, simultaneously.
\end{prop}
\begin{proof} First assume that $\Ext^1(f, -)$ vanishes over $\Gp(R)$. Hence, as observed in the proof of Theorem \ref{thm}, $f$ factors through a projective module. Now since projectives are also  injective objects of $\Gp(R)$, one deduces that $\Ext^1(-, f)=0$ over $\Gp(R)$. For the  converse, {one may take}  a syzygy sequence $\beta: 0\rt M\rt Q\rt\syz^{-1}M\rt 0$ in $\Gp(R)$. By our assumption, the morphism $\Ext^1(\syz^{-1}M, f): \Ext^1_R(\syz^{-1}M, M)\rt\Ext^1_R(\syz^{-1}M, N)$ is zero. This leads us to infer that the push-out of $\beta$ along $f$,  $f\beta$, splits. Hence $f$ factors through $Q$, and so,  $\Ext^1(f, -)=0$. Thus the proof is completed.
\end{proof}

{\bf Proof of Theorem \ref{thm2}.} $(1)\Rightarrow (2)$
By the assumption,  $\Tor_1(f, -)=0$ on $\Gp( R^{op})$. We would like to prove that $\Ext^1(f, -)=0$ over $\Gp(R)$. Assume that $X$ is a Gorenstein projective $R^{op}$-module and  ${\rm{P}}^{\bullet}:\cdots\rt  P^{-1}\rt P^0\rt P^{1}\rt \cdots$ is its complete projective resolution. By our hypothesis, the morphism $\H_{-1}(f\otimes_R{\rm{P}}^{\bullet}):\H_{-1}(M\otimes_R{\rm{P}}^{\bullet})\lrt\H_{-1}(N\otimes_R{\rm{P}}^{\bullet})$ is zero. Consider the following commutative diagram:

\[\xymatrix{ M\otimes_R{\rm{P}}^{\bullet}\ar[r]\ar[d]& N\otimes_R{\rm{P}}^{\bullet}\ar[d]\\  \hom_{R^{op}}(M^*, R)\otimes_R{\rm{P}}^{\bullet}\ar[r]\ar[d]& \hom_{R^{op}}(N^*, R)\otimes_R{\rm{P}}^{\bullet}\ar[d] \\  \hom_{R^{op}}(M^*, R\otimes_R{\rm{P}}^{\bullet})\ar[r]\ar[d]& \hom_{R^{op}}(N^*, R\otimes_R{\rm{P}}^{\bullet})\ar[d]\\  \hom_{R^{op}}(M^*, {\rm{P}}^{\bullet})\ar[r] & \hom_{R^{op}}(N^*, {\rm{P}}^{\bullet}),}\]where  $(-)^*:=\hom_{R}(-, R)$. One should note that since $M$ and $N$ are Gorenstein projective $R$-modules, the top vertical morphisms will be isomorphism. Moreover, the middle vertical morphisms are isomorphism, thanks to \cite[Theorem 3.2.14]{ej}. Finally, the bottom vertical maps are trivially isomorphism. This observation together with our hypothesis leads us to infer that
the morphism $\H^{-1}(\hom_{R^{op}}(f^*, {\rm{P}}^{\bullet})):\H^{-1}(\hom_{R^{op}}(M^*, {\rm{P}}^{\bullet}))\lrt\H^{-1}(\hom_{R}(N^*, {\rm{P}}^{\bullet}))$ is zero, as well. Since $M, N\in\Gp(R)$, it is plain that $M^*, N^*\in\Gp(R^{op})$. On the other hand,
since in the category of Gorenstein projective modules, projectives and injective are the same, the latter morphism is identified with the morphism of abelian groups $\Ext^1_{R^{op}}(M^*, \syz^3X)\rt\Ext^1_{R^{op}}(N^*, \syz^3X)$, and in particular, it will be zero. Now, as $X$ was arbitrary, by letting $X:=\syz^{-2}M^*$, we deduce that $\Ext^1(f^*, \syz M^*): \Ext^1_{R^{op}}(M^*, \syz M^*)\rt\Ext^1_{R^{op}}(N^*, \syz M^*)$ is zero. Hence, as we have seen in the proof of Theorem \ref{thm}, the morphism of $R^{op}$-modules  $f^*:N^*\rt M^*$ factors through a projective $R^{op}$-module, say $Q$. In particular, another use of the fact that $M, N\in\Gp(R)$, one concludes that
the morphism $f$ also factors through the projective $R$-module $Q^*$. Consequently, $\Ext^1(f, -)=0$, as derired.\\  $(2)\Rightarrow (3)$  This follows from Proposition \ref{prop}.\\ $(3)\Rightarrow (1)$ By our assumption, $\Ext^1(-, f)=0$ over $\Gp( R)$. According to the proof of Proposition \ref{prop}, the morphism $f$ factors through a projective $R$-module, implying  that the natural transformation  $\Tor_1(f, -)=0$  over $\Gp(R^{op})$. So the proof is finished. \ \ \ \ \ \ \ $\Box$ \\

{
As an interesting consequence of Theorem \ref{thm2}, we include the result below. {First, we give an auxiliary lemma.
\begin{lem}\label{lem}Let $f:M\rt N$ be a morphism in $\Gp(R)$. Then there exists a short exact sequence $0\rt M\st{[f~~u]}\rt N\oplus P\rt C\rt 0$ in $\Gp(R)$, where $P$ is projective.
\end{lem}
\begin{proof}Since $M$ is Gorenstein projective, we may take a short exact sequence $0\rt M\st{u}\rt P\rt L\rt 0$ in $\mod R$, where $P$ is projective and  $L\in\Gp(R)$. Consider the following commutative diagram with exact rows and columns:{\footnotesize\[\xymatrix{&& 0\ar[d]&0\ar[d]\\&& N\ar@{=}[r]\ar[d]_{[1~~0]^t}&N\ar[d]\\ 0\ar[r]& M \ar[r]^{[f~~u]^t} \ar@{=}[d] & N\oplus P \ar[r] \ar[d]_{[0~~1]} & C\ar[d]
\ar[r]&0\\ 0\ar[r]& M\ar[r]^{u} & P \ar[r]\ar[d] & L\ar[r]\ar[d]&0\\ && 0&0}\]}Since $\Gp(R)$ is closed under extensions,   the right-hand column implies that $C\in\Gp(R)$, and so, the upper row is the desired sequence.
\end{proof}
}

\begin{cor}\label{22}For a given morphism $f: M\rt N$ in $\Gp(R)$, the following are equivalent:
\begin{enumerate}\item  $\Ext^1(-, f)$ is epic over $\Gp(R)$. \item  $\Ext^1(f, -)$ is monic over $\Gp(R)$. \item  $\Tor_1(f, -)$ is  epic over $\Gp(R^{op})$.
\end{enumerate}
\end{cor}
\begin{proof} In view of Lemma \ref{lem},  we will have the following commutative diagram in $\Gp(R)$:
\[\xymatrix{0\ar[r]& M\ar[r]^{f'}\ar@{=}[d] & N\oplus P \ar[r]^{h} \ar[d]^{[1~~0]} & C \ar[r]&0\\ &M\ar[r]^{f}  & N }\]where the row is exact and $f'=[f~~u]^t$. Assume that $X$ is a  Gorenstein projective $R$-module. Applying the functor $\Ext^1_R(X, -)$ to this diagram yields the following commutative diagram
\[\xymatrix{\Ext^1_R(X, M)\ar[r]^{\Ext^1(X,f')}\ar@{=}[d] & \Ext^1(X, N\oplus P) \ar[r]^{\Ext^1(X,h)} \ar[d] & \Ext^1_R(X, C)\\ \Ext^1_R(X, M)\ar[r]^{\Ext^1(X,f)}  & \Ext^1_R(X, N) }\]in which the row is exact. As $X$ is Gorenstein projective, $\Ext^1(X, f)=\Ext^1(X, f')$. So $\Ext^1(X, f)$ is an epimorphism if and only if $\Ext^1(X, h)=0$. Similarly, applying the functor $\Ext^1(-, X)$ to the former diagram yields that $\Ext^1(h, X)=0$ if and only if $\Ext^1(f, X)$ is a monomorphism. Also, for a given $R^{op}$-module $Y$, one applies the functor $\Tor_1(-, Y)$ to the former diagram, and deduce that $\Tor_1(f, Y)$ is an epimorphism if and only if $\Tor_1(h, Y)=0$. On the other hand, by Theorem \ref{thm2}, the natural transformations $\Ext^1(h, -)$, $\Ext^1(-, h)$ and $\Tor_1(h, -)$ vanish, simultaneously. All of these give the desired result.
\end{proof}
{
\begin{cor}\label{222}Let $f: M\rt N$ be a morphism in $\Gp(R)$. Then the following are equivalent:
\begin{enumerate}\item  $\Ext^1(-, f)$ is monic over $\Gp(R)$. \item  $\Ext^1(f, -)$ is epic over $\Gp(R)$. \item  $\Tor_1(f, -)$ is monic over $\Gp(R^{op})$.
\end{enumerate}
\end{cor}
\begin{proof}Consider the following diagram of Gorenstien projective $R$-modules:
\[\xymatrix{&  & M \ar[r]^{f} \ar[d]^{[1~~0]^t} & N \ar@{=}[d]\\ 0\ar[r]&K\ar[r]^{h}  & M\oplus P \ar[r]^{f'} & N\ar[r]& 0,}\]where the row is exact and $P\st{\pi}\rt N$ is an epimorphism with $P$ projective. One should note that since $\Gp(R)$ is closed under the kernel of epimorphisms, $K\in\Gp(R)$. Since for a given object $X\in\Gp(X)$, the isomorphism $\Ext^1_R(X, M)\cong\Ext^1_R(X, M\oplus P)$ holds, applying the functor $\Ext^1_R(X, -)$ to this diagram, leads us to infer that $\Ext^1(X, f)$ is a monomorphism if and only if $\Ext^1(X, h)=0$. Similarly, one deduces that $\Ext^1(f, X)$ is an epimorphism if and only if $\Ext^1(h, X)=0$. Also, for a given Gorenstein projective $R^{op}$-module $Y$, one may apply the functor $\Tor_1^R(-, Y)$ to the above diagram and conclude that $\Tor_1(f, Y)$ is a monomorphism if and only if the morphism$\Tor_1(h, Y)=0$. Now Theorem \ref{thm2} completes the proof.
\end{proof}

Combining the preceding two corollaries, we immediately derive the result below.
\begin{cor}\label{gor} For a given morphism $f: M\rt N$ in $\Gp(R)$, the following are equivalent:
\begin{enumerate}\item  $\Ext^1(-, f)$ is an equivalence over $\Gp(R)$. \item  $\Ext^1(f, -)$ is an equivalence over $\Gp(R)$. \item  $\Tor_1(f, -)$ is an equivalence over $\Gp(R^{op})$.
\end{enumerate}
\end{cor}
}
In what follows, we study the vanishing of natural transformations of (co)homology functors over an $n$-Gorenstein ring $R$, i.e. a left and right noetherian ring with self-injective dimension at most $n$ on both sides for some non-negative integer $n$.

\begin{prop}\label{co}Let $R$ be an $n$-Gorenstein ring and let $f: M\rt N$ be a morphism in $\mod R$. Then for a given integer  $i>n$, the following are equivalent:
\begin{enumerate}\item $\Ext^{i}(f, -)$ vanishes over $\mod R$. \item $\Ext^{i}(-, f)$ vanishes over $\mod R$. \item $\Tor_{i}(f, -)$ vanishes over $\mod R^{op}$.
\end{enumerate}
\end{prop}
\begin{proof}  
 Since the implications $(1) \Leftrightarrow (3)$ are true by Corollary \ref{i1}, it suffices to examine the validity of the implications  $(1)\Leftrightarrow (2)$.\\
$(1) \Rightarrow (2)$  Assume that $\Ext^i(f, -)=0$ over $\mod R$. We would like to show that the same is true for $\Ext^i(-, f)$. {If $i=1$, then $n=0$ and so, every module should be  Gorenstein projective. Thus in this case the result follows from Proposition \ref{prop}. Now assume that $i\geq 2$}. By dimension shifting, one has $\Ext^i(f, -)\cong\Ext^1(\syz^{i-1}f, -)$. {This together with our hypothesis} implies that $\Ext^1(\syz^{i-1}f, -)=0$ over $\mod R$. Thus, as observed in the proof of Theorem \ref{thm}, $\syz^{i-1}f$ factors through a projective module. Since $R$ is $n$-Gorenstein, by using \cite[Theorem 3.1.17]{ej} one infers that the supremum of the injective length (dimension) of projective $R$-modules, denoted by $\mathsf{silp} R$, is at most  $n$.  Hence, for any integer
$m>n$, the natural tranformation $\Ext^m(-, \syz^{i-1}f)$ vanishes over $\mod R$. In particular, $\Ext^{2i-1}(-, \syz^{i-1}f)=0$ over $\mod R$. Now another use of the fact that $\mathsf{silp}R\leq n$, enables us to infer that $\Ext^{2i-1}(-, \syz^{i-1}f)\cong\Ext^i(-, f)$, and so, $\Ext^i(-, f)=0$ over $\mod R$, as desired.\\
$(2) \Rightarrow (1)$ Since $\mathsf{silp}R\leq n$ and  $i>n$, we have $\Ext^i(-, f)\cong\Ext^{2i-1}(-, \syz^{i-1}f)$. This, in conjunction with our assumption, yields that $\Ext^{2i-1}(-, \syz^{i-1}f)=0$ over $\mod R$. Particularly, $\Ext^{2i-1}(\syz^{-2i+1}\syz^{i-1}M, \syz^{i-1}f)=0$. One should note that since $R$ is $n$-Gorenstein and $i>n$, $\syz^{i-1}M\in\Gp(R)$, see \cite[Theorem 10.2.14]{ej}. Consequently, the morphism of abelian groups $\Ext^1(\syz^{i-2}M, \syz^{i-1}f): \Ext^1(\syz^{i-2}M, \syz^{i-1}M)\lrt\Ext^1(\syz^{i-2}M, \syz^{i-1}N)$ is zero.
Thus, as we have seen in the proof of Proposition \ref{prop}, $\syz^{i-1}f$ factors through a projective module. This yields that  $\Ext^1(\syz^{i-1}f, -)=0$  over $\mod R$, and so, the same is true for the natural transformation $\Ext^i(f, -)$. Hence the proof is completed.
\end{proof}

\begin{cor}\label{ccc}Let $R$ be an $n$-Gorenstein ring and $f:M\rt N$ a morphism in $\mod R$. Then  for a given integer  $i>n$, the following assertions are equivalent:
\begin{enumerate}\item  $\Ext^{i}(-, f)$ is epic. \item  $\Ext^{i}(f, -)$ is monic. \item  $\Tor_{i}(f, -)$ is epic.
\end{enumerate}
\end{cor}
\begin{proof}Since $R$ is $n$-Gorenstein, by \cite[Theorem 10.2.14]{ej} we have that $\Gpd_RM\leq n$. So there exists a short exact sequence of $R$-modules $0\rt M\st{u}\rt P\rt G\rt 0$, where $P$ has projective dimension at most $n$ and $G\in\Gp(R)$, see \cite[Proposition 1.2]{et} and also \cite[Lemma 2.17]{cfh}. Consider the following commutative diagram with the exact row:
\[\xymatrix{0\ar[r]& M\ar[r]^{f'}\ar@{=}[d] & N\oplus P \ar[r]^{h} \ar[d]^{[1~~0]} & C \ar[r]&0,\\ &M\ar[r]^{f}  & N }\]where $f'=[f~~u]^t$. Take an arbitrary $R$-module $X$. Since $\mathsf{silp} R\leq n$ and $i>n$, one has the isomorphisms $\Ext^i_R(X, N)\cong\Ext^i_R(X, N\oplus P)$ and $\Ext^i_R(N, X)\cong\Ext^i_R(N\oplus P, X)$. Similarly, for a given $R^{op}$-module $Y$, the isomorphism $\Tor_i^R(N, Y)\cong\Tor_i^R(N\oplus P, Y)$ holds true, as well. Now following the argument given in the proof of Corollary \ref{22} and applying Proposition \ref{co}, would yield the desired result.
\end{proof}

\begin{cor}\label{cccc} Let $R$ be an $n$-Gorenstein ring and $f:M\rt N$  a morphism in $\mod R$. Then for a given integer $i>n$, the following are equivalent:
\begin{enumerate}\item  $\Ext^{i}(-, f)$ is monic. \item $\Ext^{i}(f, -)$ is epic. \item  $\Tor_{i}(f, -)$ is monic.
\end{enumerate}
\end{cor}
\begin{proof} Since $R$ is $n$-Gorenstein, $\mathsf{silp}R\leq n$, and so, for any integer $i>n$ and an $R$-module $X$, the isomorphism $\Ext^i_R(X, M)\cong\Ext^i_R(X, M\oplus P)$ holds.  Consider Diagram \ref{ee} that appeared in the proof of Corollary \ref{222}.
Now the verbatim argument that is used in the proof of Corollary \ref{222}, and applying Proposition \ref{co} would yield the desired result. So the proof is finished.
\end{proof}

Combining Corollaries \ref{ccc} and \ref{cccc} yields the result below.
\begin{cor}Let $R$ be an $n$-Gorenstein ring and $f:M\rt N$  a morphism in $\mod R$. Then for a given integer $i>n$,
the following are equivalent:
\begin{enumerate}\item $\Ext^{i}(f, -)$ is an equivalence. \item $\Ext^{i}(-, f)$ is an equivalence. \item $\Tor_{i}(f, -)$ is an equivalence.
\end{enumerate}
\end{cor}
}

We close the paper with the following remark.
\begin{rem}\label{s100}{As mentioned in the introduction, the vanishing of the natural transformations of (co)homology functors have a close connection with the notions of phantom and $\Ext$-phantom morphisms.} The concept of phantom morphisms has its roots in topology in the study of maps between CW-complexes \cite{mc}. A map $f: X \rt Y$ between CW-complexes is said to be a phantom map if its restriction to each skeleton $X_n$ is null homotopic. Later, this notion has been used in various settings of mathematics. In the context of triangulated categories, phantom morphisms were first studied by Neeman \cite{ne}. The notion of phantom morphisms was also developed in the stable category of a finite group ring in a series of works by Benson and Gnacadja \cite{gn, be2, be1, be}.  The definition of a phantom morphism was generalized to the category of $R$-modules over an associative ring $R$ by Herzog \cite{he}. Precisely, a morphism $f: M\rt N$ of (not necessarily finitely generated) right $R$-modules is called a phantom morphism if the natural transformation $\Tor_1(f, -):\Tor^R_1(M, -)\rt\Tor^R_1(N, -)$  is zero, or equivalently, the pullback of any short exact sequence along $f$ is pure exact.  Similarly, $f$ is called an $\Ext$-phantom morphism, if for any finitely presented $R$-module $X$, the induced morphism $\Ext^1(X, f):\Ext^1_R(X, M)\rt\Ext^1_R(X, N)$ is zero.  So Theorem \ref{thm2} reveals that over the category of Gorenstein projective modules, the notions of phantom morphisms and $\Ext$-phantom morphisms coincide.
\end{rem}




\end{document}